\documentclass[12pt]{amsart}
\usepackage{graphicx}
\usepackage{amssymb}
\usepackage{amsmath}
\usepackage{amsfonts}
\usepackage{amsthm}
\usepackage[english]{babel}
\usepackage[latin1,ansinew]{inputenc}
\usepackage[square,comma,numbers]{natbib}
\usepackage{a4wide}
\usepackage{color}
\usepackage{academicons}

\newcommand{\be}{\begin{eqnarray}}
	\newcommand{\ee}{\end{eqnarray}}

\swapnumbers
\newtheorem{theo}{Theorem}[section]
\newtheorem{remark}[theo]{Remark}
\newtheorem{lemma}[theo]{Lemma}

\newtheorem{defi}[theo]{Definition}
\newtheorem{prop}[theo]{Proposition}
\newcommand{\R}{\mathbb R}
\newcommand{\C}{\mathbb C}
\newcommand{\N}{\mathbb N}

\newcommand{\mD}{{\mathcal D}}
\newcommand{\mB}{{\mathcal B}}
\newcommand{\mL}{{\mathcal L}}

\newcommand{\mN}{{\mathcal N}}
\newcommand{\mM}{{\mathcal M}}

\newcommand{\mJ}{{\mathcal J}}
\newcommand{\mH}{{\mathcal H}}
\newcommand{\mT}{{\mathcal T}}
\newcommand{\mU}{{\mathcal U}}

\newcommand{\mS}{{\mathcal S}}
\newcommand{\mF}{{\mathcal F}}

\newcommand{\lxi}{{\langle \xi \rangle}}

\newcommand{\eps}{\epsilon}

\newcommand{\op}{\operatorname{op}}
\newcommand{\Op}{\operatorname{Op}}
\newcommand{\Rep}{\operatorname{Re}}

\newcommand{\beq}{\begin{equation}}
	\newcommand{\eeq}{\end{equation}}

\numberwithin{equation}{section}

\begin{document}
	
	\title[Global existence of small solutions for general hyperbolic balance laws]{Global existence and asymptotic decay for small solutions of general quasilinear hyperbolic balance laws}
	\author[Sroczinski]{Matthias Sroczinski}
	\address{University of Konstanz, Konstanz, DE}
	\email{matthias.sroczinski@uni-konstanz.de}
	\thanks{This work was supported under DFG Grants nos.\ FR 822/10-2 and FR 822/11-1 (SPP-2410).}

\begin{abstract}
	
	This paper establishes global existence and asymptotic decay for small solutions to quasilinear systems of hyperbolic balance laws, where, generalizing previous works, the hyperbolic operator does not need to admit an entropy nor does the source term need to satisfy any symmetry assumptions. Dissipative properties are characterized by three conditions corresponding to regimes of small, intermediate and large wave numbers in Fourier space and the fully non-linear system is treated by using methods of para-differential calculus recently developed in the context for proofs of global existence and decay in second-order hyperbolic systems. The present work leads, in particular, to  asymptotic stability of rest-states for multidimensional  Jin-Xin relaxation system, a result not accessible through previous methods. \\
	
	\vspace{-\baselineskip}
	
\end{abstract} 

\maketitle

\tableofcontents

\section{Introduction}
The descriptions of non-equilibrium processes in physics as well as numerical schemes for conservation laws are often based on hyperbolic balance laws, which are systems of partial differential equations of the form
\begin{equation} 
	\label{eq:haupt}
	U_t(t,x)+\sum_{j=1}^df^j(U(t,x))_{x_j}=Q(U(t,x)),~~(t,x) \in [0,T] \times \R^d
\end{equation}
to be solved for $U:[0,\infty) \times \R^d \to \R^n$ taking values in some domain $\mU \subset \R^{n}$ (the state space) and $f^1,\ldots,f^d,Q: \mU \to \R^n$ are given smooth vector-fields with
$$Q(U)=\begin{pmatrix}
	0_{\R^m} \\ q(U)
\end{pmatrix},~~\text{for}~~q:\mathcal U \to \R^r,~r=n-m.$$   

Regarding the corresponding Cauchy problem, i.e. the combination of \eqref{eq:haupt} with an initial condition
\begin{equation}
\label{eq:haupt2}
U(0,x)=U_0(x),~~x \in \R^d, ~~U_0(\R^d)  \subset \mU.
\end{equation}
it is well-known (cf. \cite{maj84}) that, under the assumption of symmetrizability of the $Df^j$, for smooth initial data $U_0$ a unique smooth solution  to \eqref{eq:haupt}, \eqref{eq:haupt2}  exists at least up to some time finite $T>0$ and that this solution can in general not be continued for all positive times.\footnote{In fact, as soon as the left hand side is generally non-linear in some direction, there always exist smooth initial data, such that the gradient of the solution blows up in finite time \cite{bae23}.} 

In his groundbreaking PhD thesis \cite{kawa83} Kawashima presented the first  comprehensive general treatment of global in time existence for so-called small solutions of \eqref{eq:haupt}. He proved that under certain symmetry and compatibility conditions for initial data sufficiently close to a constant equilibrium state global-in-time smooth solutions exist and approach this equilibrium for $t \to \infty$, provided the coefficient matrices of the linearized equations satisfy what was later called the Kawashima-Shizuta condition \cite{kawa85}. This work was generalized in the seminal papers \cite{kawa04,yon04,kawa09} by Yong and Kawashima/Yong, respectively. The results of Kawashima and Yong have numerous very important applications, thus as the discrete velocity
models of the Boltzmann equation, Levermore's moment closure systems, Bouchut's discrete velocity BGK models, the one-dimensional system of equations of gas dynamics in thermal non-equilibrium and the three-dimensional non-linear Maxwell system from optics (cf. \cite{yon04, kawa09}).

Building on this crucial contributions, in the present work we will show that some of the aforementioned results can be embedded in an even more general framework, where the  assumptions on the interaction between  vector-fields $f^1,\ldots,f^d$ and $Q$ are considerably weaker. The theory is inspired by work on second-order hyperbolic equations \cite{FS}, \cite{sro23} likely to be extendable to hyperbolic systems of any order. As an application, we will establish that our results yield global existence and decay of small solutions for general multi-dimensional Jin-Xin models, which  do not satisfy the assumptions of \cite{kawa83,kawa04,yon04,kawa09}.

To fix some terminology consider $\bar U \in \mU$ with $Q(\bar U)=0$, i.e. a (homogeneous, constant) reference/equilibrium state of \eqref{eq:haupt}. Let $B_1,B_2 \subset H^1(\R^d,\R^n)$ be Banach-spaces and $\gamma: [0,\infty) \to [0,\infty)$, $\lim_{t \to \infty} \gamma(t)=0$.

\begin{defi}
	We call $\bar U$ asymptotically $B_1 \to B_2$ stable with rate $\gamma$ if there is $\eps>0$ such that for all $U_0$ with $U_0(\R^d) \subset \mU$,  $U_0-\bar U \in B_1$ and $\|U_0-\bar U\|_{B_1} < \eps$ there exists a unique solution $U$ of \eqref{eq:haupt}, \eqref{eq:haupt2} with 
	$$U-\bar U \in C^0([0,\infty,), B_2) \cap C^1([0,\infty), L^2)$$
	and
	$$\|U(t)-\bar U\|_{B_2} \le \gamma(t)\|U_0-\bar U\|_{B_1},~~ t \in [0,\infty).$$
\end{defi}

Using this terminology, \cite{yon04,kawa09} in particular show asymptotic $(H^s \cap L^1) \to H^{s-1}$ stability, $s \ge [d/2]+2$, with rate $(1+t)^{-\frac14}$ and the present work generalizes this result under weaker conditions. The main differences and similarities between the present work and that in \cite{kawa83,kawa85,yon01,kawa04,yon04,bia07,kawa09} are discussed in Appendix B. In both cases one of the key features is that solutions $U$ to the linearization of \eqref{eq:haupt}  at a constant rest state decay pointwise in Fourier space as
$$|\hat U(t,\xi)| \le C\exp(-c\xi^2/(1+\xi^2)t)|\hat U(0,t)|.$$

A different direction of generalizing \cite{kawa83}, \cite{yon04}, etc. consists in treating systems of so-called \textit{regularity loss type}, cf. \cite{kawa082,ued18,ued21} and the many references therein. Solutions to such systems may decay  even though the Kawashima-Shizuta conditions is violated. However the decay-rate in Fourier space is slower than proportional to $\exp(-c\xi^2/(1+\xi^2)t)$. Our results do clearly not include such systems.

It is also worth mentioning that the stability results for systems treated in \cite{yon04} have lately been extended to critical Besov-spaces \cite{kawa15,dan22}, which allow to consider a larger class of initial values (perturbations) and also characterize the decay behaviour of solutions in greater detail. The theory presented here should be a good starting point to get analogous results for general systems. 

The paper is organized as follows. In Section 2 we introduce the setting and state our main results, which are then proven in Sections 3,4,5,6. Concretely, in Section 3 we treat the linear system. In Section 4 we show our main result: non-linear asymptotic stability for small solutions. Section 5 is dedicated to illustrate the relation between the present work and the results of \cite{yon04, kawa09}. In Section 6 we apply our results to Jin-Xin relaxation systems. In Appendix A we prove some technical results corresponding to what will be called \textit{stable compatibility}. Lastly, as mentioned above, Appendix B illustrates the key differences and similarities between the present work and the consideration in \cite{kawa83,kawa85,yon01,kawa04,yon04,bia07,kawa09}.

\section{Setting and main results}

\textbf{Notation.} As usual for a Banach-space $E$ and $l \in \R$
$$H^l(\R^d,E):=\{u \in L^2(\R^d,E): \Lambda^l u \in L^2(\R^d,E)\}$$
are the $L^2$-based $E$-valued Sobolev spaces with norm
$$\|u\|_{H^l(\R^d,E)}:=\|\Lambda^l u\|_{L^2(\R^d,E)}.$$
Here $\Lambda=\mF^{-1}\langle \cdot \rangle \mF$, where $\mF$ is the Fourier-transform on $\R^d$ and $\lxi=(1+|\xi|^2)^{\frac12}$. We often just write $H^l$, $\|u\|_l$,  instead of $H^l(\R^d,E)$, $\|u\|_{H^l(\R^d,E)}$, if there is no concern for confusion, and $\|u\|$ for $\|u\|_0$. 

For a matrix $K \in \C^{N\times N}$ the matrix $K^*=\bar K^t$ is the adjoint and we call $\Rep K=\frac12(K+K^*)$
the real part of $H$.

\textbf{Assumptions on the equations}. We assume the operator on left hand side of \eqref{eq:haupt} to be  hyperbolic in the sense that for $A^j=Df^j$
\begin{enumerate} 
	\item[(H)]
	The matrix family $A(U,\xi)=\sum_{j=1}^d A^j(U)\xi_j$, $\xi=(\xi_1,\ldots,\xi_d) \in \R^d$, admits a symbolic symmetrizer, i.e. there exists a matrix family $S: C^\infty(\mU \times S^{d-1}, \C^{N \times N}) \in \C^{n \times n}$, bounded as well as all of its derivatives, and a constant $C>0$ such that for all $(U,\omega) \in \mU \times S^{d-1}$
	\begin{equation}
		\label{ss}
		S(U,\omega)=S(U,\omega)^*, \quad C^{-1}I_n \le \mS(U,\omega) \le CI_n,\quad S(U,\omega)A(U,\omega)=(A(U,\omega)S(U,\omega))^*.
	\end{equation}
\end{enumerate}
\begin{remark}
	\label{rem1}
	(H) implies that $A$ is similar to the Hermitian matrix
	$$S^{\frac12}AS^{-\frac12}=S^{-\frac12}(S A) S^{-\frac12}.$$
	In particular all eigenvalues of $A$ are real and semi-simple. On the other hand the latter  implies (H) if we additionally assume that the eigenvalues of $A(U,\omega)$ have constant multiplicity with respect to $(U,\omega) \in \mathcal U \times S^{d-1}$. In this case the operator on the right hand side of \eqref{eq:haupt} is called constantly hyperbolic. If all eigenvalues are simple it is called strictly hyperbolic.
	
	Also note that (H) already allows for a larger class of systems then the one treated in e.g. \cite{yon04}. In the latter the existence of an entropy function is assumed, which implies the simultaneous symmetrizability of $A^1,\ldots,A^d$.  
\end{remark}

As we are exclusively concerned with solutions to initial values close to a homogeneous reference state $\bar U \in \mU$, in each step of the argumentation we may restrict the state space $\mU$ to a smaller domain, still containing $\bar U$, without changing the notation. For simplicity we assume $\bar U=0$. 

Regarding the right hand side we make the following standard assumption, which specifies the considerations in the introduction:

\begin{enumerate}
	\item[(RH)]  The source term is of the form $Q(U)=\begin{pmatrix}
		0_{\R^m} \\ q(U)
	\end{pmatrix}$
	for a smooth function $q: \mU \to \R^{r}$, $r=n-m$ and all eigenvalues of $D_vq(U)$ have strictly negative real part $(U=(u,v)^t, u \in \R^m, v \in \R^r)$
\end{enumerate}

Note that (RH) in particular implies $\det q_v(0)\neq 0$. Thus for smooth solutions $U$ of \eqref{eq:haupt} in a neighbourhood of $0$ we can perform, in the same way as in \cite{yon04}, the variable transformation 
\begin{equation}
	\label{v}
	V=\phi(U)=\begin{pmatrix}
		u\\ q(u,v)
	\end{pmatrix}
\end{equation}
to obtain the equivalent system, called the normal form of \eqref{eq:haupt},
\begin{equation} 
	\label{nf}
	V_t+\sum_{j=1}^d \tilde A^j(V)V_{x_j}=\mL(V)V,
\end{equation}
where with $\psi=\phi^{-1}$
\begin{equation} 
	\label{nf2}
	\tilde A^j(V)=(D\psi(V))^{-1}A^j(\psi(V))D\psi(V), \quad \mL(V)=\begin{pmatrix} 0 & 0\\ 0 &q_v(\psi(V))\end{pmatrix}.
\end{equation}
Since (H) is invariant under diffeomorphisms and $\psi(0)=0$, we may assume w.l.o.g.  that \eqref{eq:haupt} is already given in normal form. With  $q_v(U)=L(U)$ we write \eqref{nf} as
$$\begin{pmatrix}
	u \\ v
\end{pmatrix}_t+\sum_{j=1}^d\begin{pmatrix}
	A^j_{11}(U) & A^j_{12}(U)\\
	A^j_{21}(U) & A^j_{22}(U)
\end{pmatrix}=\begin{pmatrix}
	0 & 0\\
	0 & L(U)
\end{pmatrix}\begin{pmatrix} u\\ v\end{pmatrix}$$
for matrices $A^j_{11}(U) \in \R^{m \times m}$, $A^j_{12}(U) \in \R^{m \times r}$, $A_{12}^j(U) \in \R^{n \times m}$, $A^j_{22}(U) \in \R^{r \times r}$.

\textbf{Linear estimates.} In Section 3 we will prove results on the linearization
\begin{equation}
	\label{lin}
	U_t+\sum_{j=1}^d A^j(0)U_{x_j}=\mathcal L(0)U,
\end{equation}
 which are interesting in itself and are also crucial for the non-linear stability result Theorem \ref{main}.
 
\begin{prop}
	\label{propD} Assume that $A^j$ and $\mL$ satisfy (H), (RH), as well as conditions (D1), (D2), (D3) of Section 3. Then there exist $C,c>0$ such that each solution $U \in L^2$ of \eqref{lin}  satisfies
	\begin{equation} 
	\label{od1}
	|\hat U(t,\xi)|^2 \le C\exp(-c\rho(|\xi|)t)|\hat U(0,\xi)|^2,~~(t,\xi) \in [0, \infty) \times \R^d
	\end{equation}
for $\rho(\tau)=\tau^2/(1+\tau^2)$.
\end{prop}

By standard arguments (cf. e.g. \cite{kawa83}),  \eqref{od1} implies the following decay estimate in Sobolev spaces.

\begin{prop}
	\label{prop:sobd}
	For any $s \ge 0$ there exists $C,c>0$ such that: For all $U_0 \in H^s \cap L^1$ the solution of \eqref{lin}, \eqref{eq:haupt2} satisfies the estimate
	\begin{align*}
		\|U(t)\|_s \le C(\operatorname{e}^{-ct}\|U_0\|_s+(1+t)^{-\frac{d}4}\|U_0\|_{L^1}),~~ t \in [0,\infty).
	\end{align*}
\end{prop}

If $A_{11}(0,\omega)$ as well as $A(0,\omega)$ are constantly hyperbolic (D1), (D2), (D3) are also necessary for an estimate of the form \eqref{od1} to hold.

\begin{prop}
	\label{prop:hc2}
	Let $A$ as well as $A_{11}$ be constantly hyperbolic and assume (RH). If there exist $C,c>0$ such that each solution $U \in L^2$ of \eqref{lin}  satisfies \eqref{od1}. Then (D1), (D2), (D3) are  satisfied.
\end{prop}

We also give examples which show that different from the theory in \cite{kawa83,kawa04,yon04,kawa09} decay of Fourier-Laplace modes alone, condition (D1), is in general not sufficient for the decay estimate \eqref{od1}. In fact, we have the following more general result. (Note that this will not be shown until Section 6).

\begin{prop}
	\label{prop:none}
	For each of the statements (i), (ii) there exist respective matrices $A^j \in \R^{n \times n}$, $j=1,\ldots,d$ satisfying $(H)$ and $L \in \R^{r \times r}$ with spectrum contained in the open complex left half-plane, such that (i), (ii) are true.
	\begin{enumerate}
		\item [(i)] $A^j$, $j=1,\ldots,d$, and $\mathcal L=\operatorname{diag}(0_{n \times n},L)$ satisfy (D1), (D2) but decay estimate \eqref{od1} does not hold for all solutions of \eqref{lin}.
		\item[(ii)] $A^j$ and $\mathcal L=\operatorname{diag}(0_{n \times n},L)$ satisfy (D1), (D3) but decay estimate \eqref{od1} does not hold for all solutions of \eqref{lin}.
	\end{enumerate}
\end{prop}

\textbf{Non-linear stability.}
The following is the main result of the present work. It is proven in Section 4.

\begin{theo}
	\label{main}
	Consider $d \ge 3$, $s> d/2+2$, $\bar U \in \mU$ with $Q(\bar U)=0$ and let \eqref{eq:haupt} satisfy conditions (H), (RH), (D1), (D2), of Section 3 and (D3)$'$ of Section 4 with respect to the equilibrium state $\bar U$. Then $\bar U$ is $H^s\cap L^1 \to H^{s-1}$ stable with rate $(1+t)^{-d/4}$.
\end{theo}

\textbf{Comparison with the entropy dissipative case.} In Section 5, we show that for systems that admit an entropy and satisfy the structural conditions (i)-(iii) of \cite{yon04}, which we will call entropy dissipative, the Kawashima-Shizuta condition 
\begin{enumerate}
	\item[(K)]  $\forall \omega \in S^{d-1}, \lambda \in \C:\ker((\lambda-A(0,\omega))) \cap \ker(D_UQ(0))=\{0\}$
\end{enumerate}
(structural condition (iv) in \cite{yon04}) is in fact equivalent to conditions (D1)-(D3). 

\begin{prop}
	\label{prop:ed}
	Let \eqref{eq:haupt} be entropy dissipative. Then the system satisfies (H) and (RH) and each of the conditions (D1), (D2), (D3) is equivalent to (K) and implies (D3)$'$.
\end{prop} 

On the other hand condition (D3) always implies (K). Thus  as consequence of Proposition \ref{prop:none} (ii) we obtain:

\begin{prop}
	\label{prop:k}
	There exist smooth functions $f^1,\ldots,f^d,Q: \R^n \to \R^n$, which satisfy (H), (RH) and (K) but the linearization of \eqref{eq:haupt} does not admit a decay estimate of the form \eqref{od1}.
\end{prop}

\textbf{Application to Jin-Xin relaxation systems.}
Consider a conservation law
\begin{equation} 
	\label{cl}
	u_t+\sum_{j=1}^d F^j(u)_{x_j}=0, \quad u:\R^d \to \R^m, F:\R^m \to \R^m
\end{equation}
which is Friedrichs symmetrizable, i.e. there exists a symmetric positive definite family of matrices $(s(u))_{u \in \mU}$ such that $s(U)D_uF^j(u)$ is symmetric, $j=1,\ldots,d$. Jin and Xin \cite{jin95} introduced a corresponding system of balance laws (relaxation laws) in order to examine certain numerical schemes for \eqref{cl}, namely
\begin{equation}
	\label{xc}
	\begin{aligned}
		u_t+\sum_{j=1}^dv^j_{x_j}&=0,\\
		v^j_t+b^j u_{x_j}&=-\frac{1}{\eps}(v^j-F^j(u)),\quad 1 \le j \le d,
	\end{aligned}
\end{equation}
with $u,v^j:[0,T] \times \R^d \to \R^m$, $j=1,\ldots,d$ relaxation rate $\eps>0$ and $b^j \in (0,\infty)$. For numerical purposes the latter should be chosen as small as possible, while still satisfying the dissipation condition
\begin{equation}
\label{disp1}
\big(\delta_{jk}b^j-D_uF^j(u)D_uF^k(u)\big)_{1\le j,k \le d} \ge 0.
\end{equation}

We will show that for a rest state $\bar u$ of \eqref{cl} the strict version of \eqref{disp1} at $\bar u$, namely
\begin{equation}
	\label{disp2}
	\Rep\big(\delta_{jk}b^j-D_uF^j(\bar u)D_uF^k(\bar u)\big)_{1\le j,k \le d} > 0
\end{equation}
implies (D1), (D2), (D3) for \eqref{xc} at $(u,v)=(\bar u,0)$ and thus non-linear asymptotic stability of this rest state in the sense of Theorem \ref{main} (In fact, we get the stronger result Lemma \ref{prop:jx2}). For $d=1$ the stability actually follows from \cite{yon01} and \cite{yon04}. To illustrate the contrast with the beautiful theory of Kawashima and Yong we will prove that this is not the case for $d \ge 2$ (except for $D_uf(\bar u)=0$) as the systems are not entropy dissipative for $d \ge 2$. In other words:

Each hyperbolic conservation law in $d \ge 3$ space dimensions with a rest state $\bar u$ satisfying $D_uf(\bar u)\neq 0$ leads to an example, in form of a corresponding Jin-Xin system, where our theory applies and the one of \cite{kawa83,kawa85,yon01,kawa04,yon04,kawa09} does not. 

The statements above are summarized in the following two propositions. 
\begin{prop}
	\label{prop:xc}
	Let $\bar u \in \R^m$ with $f(\bar u)=0$. Assume \eqref{disp} and $d\ge3$. Then the rest $\bar U=(\bar u,0)$ of \eqref{xc} satisfies the assertion of Theorem \ref{main}. 
\end{prop}

\begin{prop}
	\label{prop:nd}
	Let $d \ge 2$ and \eqref{xc} be entropy dissipative at a rest state $\bar U=(\bar u,0)$, $f(\bar u)=0$ (i.e. conditions (i)-(iii) in \cite{yon04} are satisfied). Then $D_uf(\bar u)=0$.
\end{prop}

\section{The linearized system}

In this Section we will prove Propositions \ref{propD} and \ref{prop:hc2}. In the following $A^j,A(\omega),\mL,...$ are always evaluated at $0$.

First we need to formulate the dissipation conditions (D1), (D2), (D3). To do this in a concise manner we introduce some terminology. 

Let $\Omega \subset \R^k$, $k \in \N$ be a smooth connected submanifold and consider a family of smooth matrices $H:\Omega \to \C^{N \times N}$, $N \in \N$, such that for each $\omega \in \Omega$ all eigenvalues $\mu_1(\omega),\ldots,\mu_{\bar l(\omega)}(\omega)$ (here the $\mu_l$ are pairwise distinct) of $H(\omega)$ are purely imaginary and semi-simple, $l=1,\ldots,\bar l$. For fixed $\bar \omega \in \Omega$  let $P_l(\omega)$ be the total projection
on the $\mu_l(\bar \omega)$-group of eigenvalues (cf. \cite{kat76}). $P_l$ is then well-defined on an neighbourhood $\mathcal V_{\bar \omega} \subset \Omega$ of $\bar \omega$. For $\omega \in\mathcal V_{\bar \omega}$ and a basis $\phi_{l1}(\omega),\ldots,\phi_{l\alpha_l}(\omega) \in \C^N$ of $R(P_l(\bar \omega))$, $\alpha_l$ being the multiplicity of $\mu_l(\bar \omega)$, we call  
$$J_{Rl}(\omega)=\begin{pmatrix} \phi_{l1}(\omega) & \ldots &\phi_{l\alpha_l}(\omega)\end{pmatrix} \in \C^{N \times \alpha_l}$$
a right projector on $R(P_l(\omega))$ and 
$$J_{Ll}(\omega)=\begin{pmatrix} \psi_{l1}(\omega)\\ \vdots\\ \psi_{l\alpha_l}(\omega) \end{pmatrix} \in \C^{\alpha_l \times N},$$
the corresponding left projector, where $\psi_{l1},\ldots,\psi_{l\alpha_l} \in \C^{1 \times N}$ is the dual basis to $\phi_{l1},\ldots,\psi_{l\alpha_l}$. Above and in the following, for a linear operator $T$,  $R(T)$ is the image of $T$. In Definition \ref{def:cs} and Lemmas \ref{lem:ch}-\ref{lem:sc} we always assume that $H$ has the properties indicated above.
\begin{defi}
	\label{def:cs}
	Consider $h:\Omega \to \C^{N \times N}$ smooth. We say that $h$ is stably compatible with $H$ if the following holds:
	\begin{enumerate}
		\item[(SC)]
	 For all $\bar \omega \in \Omega$,  there exists a neighbourhood $\mathcal V_{\bar \omega} \subset \Omega$ of $\bar \omega$ and for all $l=1,\ldots, \bar l$ a smooth right projector $J_{Rl}:\mathcal V_{\bar \omega} \to \C^{N \times \alpha_l}$ 
	 such that for all $\omega \in \mathcal V_{\bar \omega}$
	 \begin{equation}	 	
	 	\label{proj}
	 	\Rep (J_{Ll}(\omega)H(\omega)J_{Rl}(\omega)) = 0, \quad \Rep(J_{Ll}(\bar \omega)h(\bar \omega)J_{Rl}(\bar \omega)) <0.
	 \end{equation}
	\end{enumerate}
\end{defi}

Next, we state three lemmas corresponding to Definition \ref{def:cs}.  The first two show that the situation is simpler in two special cases  while the third one is crucial for formulating the dissipation conditions in terms of stable compatibility. As to not disrupt the main line of our argumentation both Lemmas will be proven in  Appendix A.
\begin{lemma}
	\label{lem:ch}
	Suppose the multiplicity of each eigenvalue of $H(\omega)$ does not depend on $\omega \in \Omega$. Then $h:\Omega \to \C^{N \times N}$  is stably compatible with $H$ if:
	\begin{enumerate}
		\item[(SC)$'$] For all $\bar \omega \in \Omega$, $l=1,\ldots,\bar l$, and each projector $J_{Rl}:\mathcal V_{\bar \omega} \to \C^{N \times \alpha_l}$ on $R(P_l(\omega))$ all eigenvalues of  the matrix $J_{Ll}(\bar \omega)h(\bar \omega)J_{Rl}(\bar \omega)$ have strictly negative real part.
	\end{enumerate}
\end{lemma}

\begin{lemma}
	\label{lem:sym}
	Consider smooth families of matrices $H:\Omega \to \C^{N \times N}$ and $h:\Omega \to \C^{N \times N}$, such that $H(\omega)$ is anti-Hermitian and for all $\bar \omega \in \Omega$, $l=1,\ldots,\bar l$ and for all unitary right projectors $J_{Rl}(\omega):\mathcal V_{\bar \omega} \to \C^{N \times \alpha_l}$, $J_{Ll}(\omega)=J_{Rl}(\omega)^*$, it holds
	$$J_{Ll}(\bar \omega)h(\bar \omega)J_{Rl}(\bar \omega)=(J_{Ll}(\bar \omega)h(\bar \omega)J_{Rl}(\bar \omega))^*.$$
	Then $h$ is stably compatible with $H$ if and only if (SC)$'$ holds.
\end{lemma}

\begin{lemma}
	\label{lem:sc}
 Consider a smooth function $\mathcal H:[0,\infty) \times \Omega \to \C^{N \times N}$ and set $H=\mathcal H(0,\cdot))$, $h=(\partial_1 \mathcal H)(0,\cdot)$. Assume that that for each $\omega \in \Omega$ all eigenvalues $\mu_1(\omega),\ldots,\mu_{\bar l(\omega)}(\omega)$ of $H(\omega)$ are purely imaginary and semi-simple, $l=1,\ldots,\bar l$, and that $h$ is stably compatible with $H$. Then for each compact set $K \subset \Omega$ there exist $\kappa_0>0$ and a smooth family of matrices $\mathcal D:[0,\kappa_0] \times K \to \C^{n \times n}$  such that for some constants $C,c>0$ and all $(\kappa,\omega) \in [0,\kappa_0] \times K$
	\begin{align}
		\label{S1}
		\mathcal D(\kappa,\omega)=\mathcal D(\kappa,\omega)^*,& \quad C^{-1}I_N\le \mathcal D(\kappa,\omega) \le CI_N,\\
		\label{S2}
		\Rep (\mathcal D(\kappa,\omega)\mathcal H(\kappa,\omega)) &\le -c\kappa I_N.
	\end{align}
\end{lemma}

We are now in the situation to formulate the dissipation conditions.
\begin{itemize}
	\item[(D1)] For all $\xi \in \R^d\setminus\{0\}$ all solutions $\lambda(\xi) \in \C$ of the dispersion relation
	\begin{equation} 
	\label{disp}
	\lambda(\xi)+iA(\xi)-\mL=0.
	\end{equation}
	have strictly negative real-part.
		\item[(D2)] 
	 For all $\omega \in S^{d-1}$ each eigenvalue of $A_{11}(\omega)$ is real, semi-simple and the family 
	 $$(A_{12}(\omega)L^{-1}A_{21}(\omega))_{\omega \in S^{d-1}}$$ is stably compatible with $(-iA_{11}(\omega))_{\omega \in S^{d-1}}$
	\item[(D3)]
	 $\mathcal L$ is stable compatible with  $(-iA(\omega))_{\omega \in S^{d-1}}$.
\end{itemize}

Regarding (D1), (D2), (D3) we want to point 
out the following important simplification, which can be deduced from the proof of Proposition \ref{propD} below.

\begin{remark}
	If (D2) or (D3) hold, then (D1) is equivalent to the fact that for all $\xi \in \R^{d}\setminus\{0\}$ each solution $\lambda(\xi)$ to \eqref{disp} has non-vanishing real part.
\end{remark}

To show Proposition \ref{propD} we prove the following equivalent statement about the symbol  of \eqref{lin}.  $$\mM(\xi)=-iA(\xi)+\mathcal L.$$
Note that the equivalence follows from the Kreiss Matrix Theorem \cite{kre59}.
\begin{lemma}
	\label{lem:help}
	There exists a family of matrices $\R^d\ni \xi \mapsto \mathcal D(\xi) \in \C^{N \times N}$,  and constants $C,c>0$ such that for all $\xi \in \R^d$
	\begin{equation}
		\label{t1}
		\mathcal D(\xi)=\mathcal D(\xi)^*,~~C^{-1}I_n \le\mathcal D(\xi) \le CI_n \quad \Rep (\mathcal D(\xi)\mathcal M(\xi)) \le -c\rho(|\xi|)I_n.
	\end{equation}
\end{lemma}

For the proof closely follow \cite{FS} (see also \cite{sro23}). It is split into 3 parts.

\begin{lemma}
	\label{lem:mw}
	For all $r_1, r_2$ (D1) implies the assertion of Lemma \ref{lem:help} restricted to $\xi \in \R^d$ with $r_1 \le |\xi| \le r_2$ (with constants depending on $r_1,r_2$) .
\end{lemma}

\begin{lemma}
	\label{lem:sw}
	There exists $r_0>0$ such that (D2) implies the assertion of Lemma \ref{lem:help} restricted to $\xi \in \R^d$ with $|\xi| \le r_0$ (with constants depending on $r_0$) .
\end{lemma}

\begin{lemma}
	\label{lem:lw}
	There exists $r_\infty>0$ such that (D3) implies implies the assertion of Lemma \ref{lem:help}  restricted to $\xi \in \R^d$ with $|\xi| \ge r_\infty$ (with constants depending on $r_\infty)$.
\end{lemma}

\begin{proof}[Proof of Lemma \ref{lem:mw}]
	This follows directly from the fact that the solutions of \eqref{disp} are exactly the eigenvalues of $\mM(\xi)$, the principle of linearized stability and compactness of $[r_1,r_2] \times S^{d-1}$, see e.g. \cite{FS} for more details.
\end{proof}

\begin{proof}[Proof of Lemma \ref{lem:sw}]
	We use polar coordinates $\R^d \ni \xi=\kappa\omega$, $\kappa=|\xi|, \omega \in S^{d-1}$. First, fix $\bar \omega \in S^{d-1}$ and consider a neighbourhood
	$$P_{\delta}(\bar \omega):=[0,\delta] \times B(\bar \omega,\delta)$$
	of $(0,\bar \omega) \in [0,\infty) \times S^{d-1}$ for some $\delta>0$. 
	In unpublished lecture notes \cite{fre23} Freistühler  proved that there exists $\delta,C>0$ and a familiy of uniformly bounded basis transformations $\mathcal R:P_\delta \to \C^{N \times N}$,  such that for all $(\kappa,\omega) \in P_{\delta}$
	\begin{align*}\mathcal R(\kappa,\omega)\mM(\kappa,\omega)\mathcal R(\kappa,\omega)^{-1}&=\begin{pmatrix}
			\kappa M_1(\kappa,\omega) & 0\\
			0 & M_2(\kappa,\omega)
		\end{pmatrix}
	\end{align*}
	for matrices $M_1(\kappa,\omega) \in C^{m \times m}$, $M_2(\kappa,\omega) \in \C^{r \times r}$ with
	$$M_1(0,\omega)=-iA_{11}(\omega),~~(\partial_\kappa M_1)(0,\omega)=A_{12}(\omega)L^{-1}A_{21}(\omega),~~M_2(0,\omega)=L.$$
	By condition (D2) and Lemma \ref{lem:sc}, after possibly choosing a smaller $\delta>0$, there exists a smooth matrix family $\mD_1:P_\delta \to \C^{m \times m}$,  satisfying \eqref{S1}, \eqref{S2} with  $\mathcal H$ replaced by $M_1$ and by (RH) there exists $\mD_2 \in \C^{r \times r}$, $c_2>0$ such that $\mD_2$ satisfies \eqref{S1} and $\mD_2L_2<-c_2I_r$. This implies that for $\tilde \mD=\operatorname{diag}(\mD_1,\mD_2)$  restricted to $\xi=r\omega$, $(r,\omega) \in P_{\delta/2}$ Lemma \ref{lem:help} is satisfied with
	$$\mD(r\omega)=R^*(r,\omega)\tilde \mD(r,\omega) R(r,\omega).$$
	We can construct a matrix family $\mD$ satisfying Lemma \ref{t1} restricted to all $|\xi| \le r_0$, $r_0$ sufficiently small, by partitioning $\Omega=S^{d-1}$ as in the proof of Lemma \ref{lem:sc}.
\end{proof}

\begin{proof}[Proof of Lemma \ref{lem:lw}]
	For $\nu>0$, $\omega \in S^{d-1}$ consider the matrix family
	$$\mN(\nu,\omega)=\nu \mM(1/\nu\omega)=-iA(\omega)+\nu L$$
	By (D3) and Lemma \ref{lem:sc} we find a smooth matrix family $\tilde \mD:[0,r_3] \times S^{d-1} \to \C^{n \times n}$ satisfying \eqref{S1}, \eqref{S2} for $M=\mN$.
	By definition this implies the assertion with $\mD(\kappa\omega)=\tilde \mD(1/\kappa,\omega)$ and $r_\infty=1/r_3$.
\end{proof}

Using a partition of unity on $\R^d$ with respect to the three regions $|\xi|  \le r_0$, $|\xi| \in [r_0,r_\infty]$, $|\xi| \ge r_\infty$ similar to the one in Lemma \ref{lem:sc} we can then easily construct a matrix family $\mD(\xi)$ such that Lemma \ref{lem:help} is satisfied for all $\xi \in \R^d$. 
	 
It remains to show Proposition \ref{prop:hc2} 

\begin{proof}[Proof of Proposition \ref{prop:hc2}]
	We have seen in the proofs of Lemmas \ref{lem:sw} and \ref{lem:lw} that the eigenvalues $\lambda(\xi)$ of $\mathcal M(\xi)$ for $|\xi| \to 0$ admit the expansion $\omega=\xi/|\xi|$
	$$\lambda(\xi)=-i|\xi|\mu_l(\omega)+|\xi|^2\zeta_l(\omega)+o(|\xi|^2),~\text{or } \lambda(\xi)=\nu+o(|\xi|),$$
	whith $\mu_l(\omega)$ an eigenvalue of $A_{11}(\omega)$, $\zeta_l(\omega)$ an eigenvalue of  $$J_{lL}(\omega)A_{12}(\omega)L^{-1}A_{21}(\omega)J_{Rl}(\omega),$$
	 $\nu$ an eigenvalue of $L$, and for $|\xi| \to \infty$
	$$\lambda(\xi)=-i|\xi|\gamma_l(\omega)+\beta_l(\omega)+o(|\xi|^{-1})$$
	with $\gamma_l(\omega)$ an eigenvalue of $A(\omega)$, $\beta_l$ an eigenvalue of $J_{lL}(\omega)\mathcal L J_{Rl}(\omega)$.
	
	Estimate \eqref{od1} is clearly equivalent to
	$$|\exp(\mathcal M(\xi)t)| \le C\exp(-c\rho(|\xi|t)), \xi \in \R^d,~t \in [0,\infty).$$ 
	By the expansions of the eigenvalues this can only be true if the real parts of all $\zeta_l$, $\beta_l$ are strictly negative. If $A_{11}$ and $A$ are constantly hyperbolic these are exactly the conditions (D2) and (D3) by Lemma \ref{lem:ch}.
\end{proof}

\section{Global existence and asymptotic decay for the quasi-linear system}

In order to treat the non-linear parts we also have to assume that (D3) does not only hold for $U=0$ but also in a neighbourhood of $0$. 
\begin{enumerate}
	\item[(D3)$'$] There exists a $0$-neighbourhood $\mU_0 \subset \mU$, such that $\mathcal L(U)$ is stably compatible with $A(U,\omega)$ on $\mU_0 \times S^{d-1}$.
\end{enumerate} 

Of course, if $A(U,\omega)$ is constantly hyperbolic this already follows from (D3) by virtue of Lemma \ref{lem:ch}. In Analogy to the linear setting we also define
$$\mM(U,\xi)=-A(U,\xi)+\mL(U).$$

As (H) implies the local well-posed of the Cauchy problem \eqref{eq:haupt}, \eqref{eq:haupt2} (cf. e.g. \cite{T91}) it suffices to show that the local solution can be extended globally for small initial data. As in \cite{kawa83} this is the case if the following holds.

\begin{prop}
	\label{prop:main}
	Let $d \ge 3$, $s>d/2+2$. Then there exist constants $\mu>0$, $\delta=\delta(\mu)>0$, and $C=C(\mu,\delta)>0$ (all independent of $T$) such that the following holds:
	For all $U_0 \in H^{s} \cap L^1 $ with 
	$\|U_0\|_{s},\|U_0\|_{L^1} <\delta$ and all $U\in C^0([0,T],H^{s})$
	satisfying \eqref{eq:haupt}, \eqref{eq:haupt2} as well as
	\begin{equation} 
		\label{bound}
		\sup_{t \in [0,T]}\|U(t)\|_{s}^2+\int_0^T \|U(\tau)\|_{s}^2~d\tau  \le \mu
	\end{equation}
		 we have for all $t \in [0,T]$
	\begin{align}
		\label{eq:decay1}
		\|U(t)\|_{s-1} &\le C(1+t)^{-\frac{d}{4}}(\|U_0\|_{s-1}+\|U_0\|_{L^1})\\
		\label{eq:decay20}
		\|U(t)\|_{s}^2+\int_0^t \|U(\tau)\|_{s}^2 d\tau &\le C(\|U_0\|_{s}^2+\|U_0\|_{L^1}^2).
	\end{align}
\end{prop}

As $t \mapsto (t+1)^{-\frac{d}4}$ is square integrable over $\R_+$ for $d \ge 3$, the proof can be split into the following two propositions.

\begin{prop}
	\label{prop:decay}
In the situation of Proposition \ref{prop:main} there exist $\mu>0$, $\delta>0$ and $C>0$ such that the following holds:
For all $U_0 \in H^{s}\cap L^1 $ with $\|U_0\|_{s}$, $\|U_0\|_{L^1} <\delta$ and all $U\in C^0([0,T],H^{s})$ 
satisfying \eqref{eq:haupt}, \eqref{eq:haupt2} and 
$$\sup_{t \in [0,T]}\|U(t)\|_{s}^2+\int_0^T \|U(\tau)\|_{s}^2~d\tau  \le \mu$$ 
\eqref{eq:decay1} holds for all $t \in [0,T]$.
\end{prop}

\begin{prop}
\label{prop:energy}
In the situation of Proposition \ref{prop:main} there exist $\mu>0$,  and $C>0$ such that the following holds:
For all $U_0 \in H^{s} $ and all $ U\in C^0([0,T],H^{s})$ 
satisfying \eqref{eq:haupt}, \eqref{eq:haupt2} and 
$$\sup_{t \in [0,T]}\|U(t)\|_{s}^2+\int_0^T \|U(\tau)\|_{s+1}^2~d\tau  \le \mu$$ we have for all $t \in [0,T]$
\begin{equation}
	\label{eq:wichtig}
	\|U(t)\|_{s}^2+\int_0^t  \|U(\tau)\|_{s}^2 d\tau
	 \le C\|U_0\|_{{s}}^2+C\int_{0}^t \|U(\tau)\|_{s-1}^2~d\tau
\end{equation}
\end{prop}

Proposition \ref{prop:decay} follows from Proposition \ref{prop:sobd} and Green's formula as in \cite{kawa83}.

For the proof of Proposition \ref{prop:energy} we closely follow \cite{sro23}. We first need to show:

\begin{prop}
	\label{prop:nl}
	Assume (D3)$'$. Then there exist $\bar r>0, \bar c>0$ and  
	a  mapping $\mD_\infty\in C^\infty(\Omega_\infty,\C^{n \times n})$, $\Omega_\infty:=\bar \mU_0 \times \{\xi \in \R^d:|\xi| \ge \bar r^{-1}\}$, $\mU_0:= \{U \in \mU: |U| < \bar r\}\subset \mU$, such that:
	\begin{enumerate} 
		\vspace{-0.5\baselineskip}
		\item[(i)] For all $(U,\xi) \in \Omega_{\infty}$ $$\mD_\infty(U,\xi)=\mD_\infty(U,\xi)^* \ge \bar c I_n,$$
		and  
		$$\Rep \mD_\infty(U,\xi)\mM(U,\xi) \le -\bar c I_{n}.$$
		\item[(ii)] For any $\alpha,\beta \in \N_0^d$ there exist $C_{\alpha\beta}>0$ with
		\begin{equation}
			\label{sym1}|\partial_U^\beta \partial_\xi^\alpha \mD_\infty (U,\xi)| \le C_{\alpha\beta}\lxi^{-|\alpha|},~~(U,\xi) \in  \Omega_\infty.
		\end{equation}
	\end{enumerate}
\end{prop}

\begin{proof}
	The proof is essentially the same as in \cite{sro23} and we just highlight the main steps here. 
	
	Consider $\mN:\mU \times (0,\infty) \times \mathbb S^{d-1} \to \C^{n \times n}$ defined by
	\begin{equation} \label{K}
		\mN(U,\nu,\omega)=\nu\mM(U,\nu^{-1}\omega)=-iA(U,\omega)+\nu \mL(U)
	\end{equation}
 	By (D3)$'$ and Lemma \ref{lem:sc} there exist $\nu_0>0, c,C>0$ and a family of Hermitian matrices $\tilde \mD:\bar \mU_0 \times [0,\nu_0] \times S^{d-1}  \to \C^{n \times n}$ such that 
 	\begin{align*}
 	C^{-1}I_n\le \mD(\nu,\omega,U) &\le CI_n,\\
 	\Rep (\mD(U,\nu,\omega)\mN(U,\nu,\omega)) &\le -cI_n
 \end{align*}	
	for all $(\nu,\omega,U) \in \mU_0 \times [0,\nu_0]\times S^{d-1}$. The assertion now follows in a straightforward manner for $\bar r=\nu_0^{-1}$ and 
	$$\mD_\infty(U,\xi)=\tilde \mD(U,|\xi|^{-1},\xi|\xi|^{-1}).$$
\end{proof}

Next we need to introduce para-differential operators. 

\begin{defi}
	For $k \in \N_0, l \in \R$, $\Gamma^l_k$ is
	the set of functions $A: \R^d \times \R^d \mapsto \C^{n\times n}$ such that,
	\begin{enumerate}
		\item[(i)]
		for almost all $x \in \R^d$ the mapping $\xi \mapsto A(x,\xi)$ is in $C^\infty(\R^d,\C^{n\times n})$
		\item[(ii)]
		
		\vspace{-0.5\baselineskip}
		for any $\alpha \in \N_0^d$ and $\xi \in \R^d$ the mapping $x  \mapsto \partial_\xi^\alpha A(x,\xi)$ belongs
		to $W^{k,\infty}(\R^d,\C^{n \times n})$ and there exists $C_\alpha > 0$ not depending on $\xi$ such that
		\begin{equation}
			\label{eq:semingamma}
			\|\partial^\alpha_\xi A(\cdot,\xi)\|_{L^{\infty}} \le C_\alpha \lxi^{l-|\alpha|}.
		\end{equation}
	\end{enumerate}
\end{defi}

\begin{defi}
	For $A \in \Gamma^l_k$ we denote by $\Op[A]:\mS(\R^d) \to \mS(\R^d)$, the para-differential operator with symbol $A$, which can be extended to a bounded linear operator from $H^{p+l}$ to  $H^p$ for all $p \in \R$. We also write $\Op[A]$ for this extension.
\end{defi}

We do not give the construction of a para-differential operator at this point and instead refer to the original works \cite{bon81}, \cite{mey81} and the very well-written textbook \cite{met08} for all the necessary definitions and results. We also need the following

\begin{defi}
	We denote by $S^l(\mU):=S^l(\mU,\C^{n \times n})$ the set of all functions $G \in C^\infty(\mU \times \R^d, \C^{n \times n})$ for which for any $\alpha,\beta \in \N_0^d$ there exists $C_{\alpha\beta}>0$ such that for all $(V,\xi) \in \mU \times \R^d$
	\begin{equation} 
		\label{eq:help*}
		|\partial_x^\beta\partial_\xi^\alpha G(V,\xi)| \le C_{\alpha\beta}\lxi^{l-|\alpha|}.
	\end{equation}
\end{defi}

For a function $G \in C^{\infty}(\mU \times \R^d,\C^{n \times n})$ and $V \in H^s$ with $V(\R^d)  \subset \mU$ we denote by $G_V$ the composition $G_V(x,\xi)=G(V(x),\xi)$. Then if $G \in S^l$ and $V \in H^s$ for $s>d/2+k$ we find $G_V \in \Gamma^l_k$ for $s>d/2+k$ (cf. ibid).

Take $s>d/2+2$ and $T>0$. Let $U \in C([0,T],H^s)$  be a solution to \eqref{eq:haupt} satisfying \eqref{bound} for some $\mu>0$. For a family of Friedrichs mollifiers   $(J_\eps)_{\eps \in (0,1)}$ we set $W:=W_\eps=\Lambda^sJ_\eps U$. We also fix $t \in [0,T]$ but suppress the dependence on this $t$ in our notation. In the following $C$ denotes a generic constant depending monotonically increasing on $\mu$.
\begin{lemma}
	\label{lem:c}
	$W$ satisfies the equation
	\begin{equation} 
		\label{eq:V}
		W_t=\Op[\mM_U]W+R_\eps(U),
	\end{equation}
	for $R_\eps(U) \in L^2$ with
	$$\|R_\eps(U)\| \le C_\mu(\|U\|_s^2+\|U\|_{s-1})$$
	for all $\eps \in (0,1)$.
\end{lemma}

\begin{proof}

First note that
$$\Op[\mM_U]=-\sum_{j=1}^d\Op[A^j_U]\partial_{x_j}+\Op[\mL_U].$$
Hence
$$U_t=\Op[\mM_U]+R_1(U),$$
where 
$$R_1(U)=\sum_{j=1}^d(A^j(U)-\Op[A^j(U)])\partial_{x_j}U+(\mL(U)-\Op[\mL_U])U.$$
By standard estimates on para-differential Operators (cf. ibid.) we find 
$$\|(A^j(U)-\Op[A^j_U])\partial_{x_j}U\|_s \le C\|\partial_{x_j}U\|_{L^\infty}\|U\|_s \le C(\|U\|_s^2+\|U\|_{s-1})$$
and the same holds for $(\mL(U)-\Op[\mL_U])U$. Thus
\begin{equation} 
	\label{eq:R1}
	\|R_1(U)\|_s \le C(\|U\|_{s}^2+\|U\|_{s-1}).
\end{equation}

Now apply $J_\eps\Lambda^s$ to \eqref{eq:haupt} to obtain 
$$W_t=\Op[\mM_U]W+R_\eps(U)$$
where
$$R_\eps(U)=[\Lambda^sJ_\eps, \Op[\mM_U]]U+J_\eps\Lambda^s R_1(U).$$
As $J_\eps$ is a uniformly bounded family of linear operators on $H^l$, $l \in \R$, \eqref{eq:R1} yields
$$\|J_\eps\Lambda^s R_1(U)\|_s \le C(\|U\|_s^2+\|U\|_{s_1}).$$
Furhtermore as $s>d/2+1$ (cf. ibid.) 
$$\|[J_\eps, \Op[A^j_U]]\partial_{x_j}U\|_s \le C\|U\|_s\|\partial_{x_j}U\|_{s-1}$$
and 
$$\|[\Lambda^s,\Op[A^j_U]]\partial_{x_j}J_\eps U\| \le C\|U\|_{s}\|\partial_{x_j}J_\eps U\|_{s-1} \le C\|U\|_s^2.$$
As this also holds for $\mL$ with $\partial_{x_j}$ replaced by $1$ the assertion follows.
\end{proof}

From this point the proof continues analogously to \cite{sro23}. There the following result is shown (although somewhat implicitly) by applying the classical results of \cite{mey81,bon81} and in particular a novel version of the strong G\r{a}rding inequality.
\begin{prop}
	\label{prop:central}
	Let $\mM \in C^{\infty}(\mU \times \R^d, \C^{n \times n})$ be a matrix family with the properties stated in Proposition \ref{prop:nl} and $X \in L^\infty([0,T], L^2(\R^d,\C^n))$.
	Then there exist $\mu, C_\mu>0$ such that for all $v \in C([0,T],H^{s}) \cap C^1([0,T], H^{s-1})$ for some $s>d/2+2$ with
	$$\|v(t)\|^2_s+\|v_t(t)\|_{s-1}^2+\int_0^t\|v(\tau)\|^2_s+\|v_t(\tau)\|_{s-1}^2 d\tau \le \mu,~~t \in [0,T]$$
	 each solutions $W \in C([0,T],H^1) \cap C^1([0,T],L^2)$ of 
	$$W_t=\op[\mM_v]W+X(v)$$
	satisfies
	\begin{equation} 
		\label{central}
		\begin{aligned}
		\|W(t)\|^2+\int_0^t \|W(\tau)\|^2 d\tau &\le C_\mu\|W(0)\|^2\\
		&\quad+C_\mu\int_0^t \|W(\tau)\|^2(\|v(\tau)\|_s+\|v_t(\tau)\|_{s-1}+\|v(\tau)\|_s^{\frac12})\\
		&\quad +\|W(\tau)\|^2_{-1} +\| X(\tau)\|\|W(\tau)\| d\tau,~~t \in [0,T]
		\end{aligned}
	\end{equation} 
\end{prop} 

Applying this proposition with $v=U, X=R_\eps(U)$  finishes the proof of Proposition \ref{prop:energy}. Indeed by definition
$W=W_\eps= \Lambda^s J_\eps U \to \Lambda^sU$ for $\eps \to 0$ uniformly with respect to $t$ and we have $\|U_t\|_{s-1} \le C_\mu \|U\|_s$ since $U$ solves \eqref{eq:haupt}.  Thus by Lemma \ref{lem:c},  \eqref{central} gives for $\mu \le 1$
\begin{equation*} 
	\label{central2}
	\begin{aligned}
		&\sup_{\tau \in [0,T]}\|U(\tau)\|^2_s+\int_0^t \|U(\tau)\|^2_s d\tau \le C_\mu\|U(0)\|^2_s+C_\mu\int_0^t\|U(\tau)\|_s^3+\|U(\tau)\|_s^{\frac52}+\|U(\tau)\|^2_{s-1}  d\tau\\
		& \le C_\mu \|U(0)\|^2_s+C_\mu \left(\sup_{\tau \in [0,T]} \|U(\tau)\|^2_s+ \int_0^t \|U(\tau)\|^2 d\tau \right)^{\frac32}+C_\mu \int_0^t \|U(\tau)\|^2_{s-1}  d\tau,
	\end{aligned}.
\end{equation*} 
Hence, if we choose $\mu$ sufficiently small \eqref{eq:wichtig} is satisfied.

\section{Relation to entropy dissipative systems}

In this section we show Propositions \ref{prop:ed} and \ref{prop:k}. We start by recalling the structural conditions of \cite{yon04}.

\begin{defi}
	We call a hyperbolic balance law \eqref{eq:haupt} (not necessarily in normal form) entropy dissipative (at $\bar U=0$) if it satisfies the following (structural conditions (i)-(iii) in \cite{yon04}, Section 1)
	\begin{enumerate}
		\item[(i)] $q_v(0)$ is invertible.
		\item[(ii)] There exists a smooth strictly convex entropy function $\eta: G \to \R$, ($0 \in G \subset \mU$ compact and convex) such that $D^2\eta(U)Df^j(U)$ is symmetric for all $U \in G$, $j=1,\ldots,d$.
		\item[(iii)] There exists a constant $c_G>0$ such that for all $U \in G$
		$$(D\eta(U)-D\eta(0))Q(U) \le -c_G|Q(U)|^2$$.
	\end{enumerate}
\end{defi}

\begin{proof}[Proof of Proposition \ref{prop:ed}]
	Clearly condition (ii) implies (H) with $S(U,\omega)=D^2\eta(U)$. 
	
	Condition (i) allows us to write \eqref{eq:haupt} in normal form \eqref{nf}, \eqref{nf2}. With $\phi$ as in \eqref{v}, $\psi=\phi^{-1}$ set $T(u):=D\psi(u)$, $T=T(0)$. As shown in \cite{yon04} (Theorem 2.1) conditions (i)-(iii) imply that there exist symmetric positive definite matrices $A^0_1 \in \R^{m \times m}, A^0_2 \in \R^{r \times r}$ such that
	$$A^0:=T^*D^2\eta(0)T=\begin{pmatrix} A^0_1 & 0\\ 0 & A^0_2\end{pmatrix}.$$ and
	\begin{equation}
		\label{t}
		T^*\left(D^2\eta(0) DQ(0)+DQ(0)^*D^2 \eta(0)\right)T=\begin{pmatrix}
			0 & 0\\
			0 & \tilde L
		\end{pmatrix}
	\end{equation} 
	with $\tilde L \in \R^{r \times r}$, $\Rep  \tilde L<0$.
	For 
	$$\mL(V)=\begin{pmatrix}
		0 & 0 \\
		0 & q_v(\psi(V))
	\end{pmatrix}$$ we have $\mL(0)=T^{-1}DQ(0)T.$
	Thus \eqref{t} reads
	$$A^0\mL(0)+\mL(0)^*A^0=\begin{pmatrix}
		0 & 0\\
		0 & \tilde L
	\end{pmatrix}$$
	In particular 
	$$\Rep A^0_2q_v(0)=\tilde L <0,$$
	which also gives
	$$\Rep (A^0_2)^{\frac12}q_v(0)(A^0_2)^{-\frac12}=(A^0_2)^{-\frac12}(\Rep \tilde L)(A^0_2)^{-\frac12}<0.$$
	We have shown that $q_v(0)$ is similar to a matrix with negative real part, which implies that all eigenvalues of $q_v(0)$ have negative real-part. As $q$ is smooth this also holds in a neighbourhood of $0$, i.e. (RH) is satisfied. Additionally, condition (ii)  yields that
	$$A^0T^{-1}A(0,\omega)T=T^*D^2\eta(0)A(0,\omega)T$$
	is symmetric and so is
	$$(A^0)^{\frac12}T^{-1}A(0,\omega)T(A^0)^{-\frac12}=(A^0)^{-\frac12}(A^0)T^{-1}A(0,\omega)T(A^0)^{-\frac12}.$$
	
	Note that (D1)-(D3) are formulated for systems in normal form (in fact, this is only relevant regarding (D2)). Thus, a-priori we have to use $\tilde A(\omega):=T^{-1}A(0,\omega)T$ instead of $A(\omega)$ in (D1), (D2), (D3). But clearly (K) and (D1)-(D3) are invariant under the basis transformation  $(A^0)^{-\frac12}$ since it is $\R^m \times \R^r$ block-diagonal. Thus we can show the equivalence for the matrices  $(A^0)^{\frac12}\mL(A^0)^{-\frac12}$ and 
	$(A^0)^{\frac12}\tilde A(\omega)(A^0)^{-\frac12}$ . In other words we may assume w.l.o.g. that
	$A(\omega)$ is symmetric and $\mL$ is of the form
	$$\mL=\begin{pmatrix}0 & 0\\ 0 & L\end{pmatrix},~~\Rep L<0.$$
	That in this case (D1) is equivalent to (K) is a standard result \cite{bia07}. 
	
	Next, we show that (D2) is equivalent to (K). As $\ker \mL=\R^m \times \{0\}$, (K) is equivalent to 
	\begin{equation}
		\label{k0}
		\ker (\lambda+A_{11}(\omega))\cap \ker (A_{21}(\omega))=\{0\},~~\lambda \in \R, \omega \in S^{d-1}.
	\end{equation}
	Since
	$$\Rep L^{-1}=(L^*)^{-1}(\Rep L)L^{-1}$$
	$\Rep L^{-1}$ is also negative definite. Now , suppose (K) holds. For an eigenvalue $\mu_l$ of $A_{11}$ choose any unitary right projector $J_{Rl}$, $J_{Ll}=(J_{Rl})^*$. The symmetry of $A$ and the definiteness of $\Re L^{-1}$ then yield
	\begin{align} 
		\nonumber
		\Rep (iJ_{Ll}A_{11}J_{Rl})&=0,\\
		\label{jl}
		\Rep\big((J_{Rl})^*A_{21}^*L^{-1}A_{21}J_{Rl}\big) &\le 0.
	\end{align}
	\eqref{k0} can also by expressed as $A_{21}J_{Rl}v \neq 0$ for all $v \in \C^{\alpha_l}\setminus\{0\},~~l=1,\ldots,\bar l.$ Then $\Rep L<0$ and \eqref{jl} imply
	$$\Rep\big((J_{Rl})^*A_{21}^*L^{-1}A_{21}J_{Rl}\big) <0.$$
	If on the other hand (D2) holds with smooth projectors $J_{Rl}(\omega)$, due to the symmetry of $A_{11}$ we can apply a smooth basis transformation $T_l(\omega) \in \C^{\alpha_l \times \alpha_l}$ such that $\tilde J_{Rl}T$ is orthonormal and (K) follows with the argumentation above. 
	
	As (K) is also equivalent to the fact that for each eigenvalue $\nu_l$ of $A$ and each right projector $\mJ_{Rl}$ it holds
	$$\ker \mL \mJ_{Rl}=\{0\},$$
	the equivalence between (D3) and (K) follows as above with $A_{11}$ replaced by $A$ and $A_{12}L^{-1}A_{21}$ by $\mL$. 
	
	Lastly, note that (ii) implies that with $A^0(u)=T(u)^*D^2\eta(u)T(u)$
	$$\check A(u,\omega)=(A^0(u))^{\frac12}(T(u)^{-1}A(u,\omega)T(A^0(u))^{-\frac12}$$
	is symmetric also in neighbourhood of $\bar U=0$. Now for some $\bar \omega \in S^{d-1}$ and an eigenvalue $\nu_l$ of $\check A(0,\bar \omega)$ let $\mJ_{Rl}(u,\omega)$, $(u,\omega)$ in some neighbourhood of $(0,\bar \omega)$, be a smooth unitary right projector corresponding to the $\nu_l$-group of eigenvalues. We have shown above that then
	$$\Rep (\mJ_{Rl}^*(0,\bar \omega)\check \mL(0)\mJ_{Rl}(0,\bar \omega)<0,$$
	with
	$$\check \mL(u)=(A^0(u))^{\frac12}(T(u)^{-1}D_uQ(u)T(A^0(u))^{-\frac12}$$
	By continuity 	
	\begin{equation} 
		\label{hl1}
		\Rep (\mJ_{Rl}^*(u, \omega)\check \mL(u)\mJ_{Rl}(u, \omega)<0,
	\end{equation}
also holds for $(u,\omega)$ in a neighbourhood of $(u,\bar \omega)$.  As $\check A(u,\omega)$ is symmetric and $\mJ_{Ll}(u,\omega)=\mJ_{Ll}^*(u,\omega)$, clearly
	\begin{equation} 
		\label{hl2}
		\Rep(-i\mJ_{Ll}(u,\omega)\check A(u,\omega)\mJ_{Rl}(u,\omega))=0.
	\end{equation}
	\eqref{hl1} and \eqref{hl2} show (D3)'.
\end{proof} 	 

\begin{proof}[Proof of Proposition \ref{prop:k}]
	We just need to show that (D3) implies (K) (for $\mL=D_UQ(0)$). Then the assertion follows from Proposition \ref{prop:none} (ii) with $Q(U)=\mL U$, $f^j(U)=A^jU$. Assume that (K) is not satisfied. Then there exist $(\lambda,\bar \omega) \in \C \times S^{d-1}$ and 
	$$0 \neq v \in \operatorname{ker}(\lambda-A(\bar \omega)) \cap \operatorname{\ker}(\mL).$$
	This implies that for a right projector on the eigenspace corresponding to $\lambda$ there exist $w \in \R^{\alpha}$, $\alpha$ being the dimension of this eigenspace, with $v=J_R(\bar \omega)w$ 
	and
	$$0=\mL v=\mL J_{R}(\bar \omega)v.$$
	Thus 
	$$\Rep (J_{L}(\bar \omega)\mL J_{R}(\bar \omega))<0$$
	does not hold. Since this is true for all right projectors, (D3) is not satisfied. 
\end{proof}

\section{Application to Jin-Xin relaxation systems}
In this section we treat Jin-Xin systems \eqref{xc}. We will not restrict ourselves to scalars $b^1,\ldots,b^d$ but allow for $b^1,\ldots,b^d$ to be symmetric positive definite matrices. With $U=(U^1,U^2,\ldots, U^{d+1})=(u,v^1,\ldots,v^d) \in \R^{m(d+1)}$ we write \eqref{xc} as 
 \begin{equation} 
 	\label{xc1}
 	U_t+\sum_{j=1}^d\tilde A^j U_{x_j}=Q(U),
 \end{equation}
where 
$$\tilde A^j=\begin{pmatrix}
	0_{m \times m} & (E^j)^t\\
	B^j & 0_{dm \times dm}
\end{pmatrix},$$
with block matrices
\begin{align*} 
	E^1=\begin{pmatrix} I_m \\ 0 \\ \vdots \\ 0 \end{pmatrix},~~ E^2&=\begin{pmatrix}  0\\ I_m \\ \vdots \\ 0 \end{pmatrix},\ldots, E^d=\begin{pmatrix} 0 \\ 0 \\ \vdots \\ I_m \end{pmatrix},\\ B^1=\begin{pmatrix}b^1 \\ 0\\ \vdots \\0\end{pmatrix},~~B^2&=\begin{pmatrix}0 \\ b^2\\ \vdots \\0\end{pmatrix},\ldots, B^d=\begin{pmatrix}0 \\ 0\\ \vdots \\b^d\end{pmatrix} \in \R^{md \times m}.
\end{align*}
and
\begin{equation} 
	\label{qu}
	Q(U)=\begin{pmatrix} 0_{m \times 1} \\q(U)\end{pmatrix} \in \R^{d(m+1)},~~q(U)=-\frac1{\eps}(v-F(u)), ~~F(u)=\begin{pmatrix}F^1(u)\\ \vdots\\F^d(u) \end{pmatrix}\in \R^{md}.
\end{equation}

The symbol of \eqref{xc1} is given as
\begin{equation} 
	\label{a}
	\tilde A(\omega)=\sum_{k=1}^d\tilde A^k\omega_k=\begin{pmatrix} 0 &\Omega^t \\ B(\omega) &0\end{pmatrix}, 
\end{equation}
with 
$$\Omega=\begin{pmatrix}
	\omega_1I_m \\ \vdots \\ \omega_dI_m
\end{pmatrix},\quad B(\omega)=\begin{pmatrix}
	\omega_1b^1 \\ \vdots \\ \omega_d b^d
\end{pmatrix},~~\omega \in S^{d-1}.$$
We also need
$$K(u,\xi)=\sum_{j=1}^d (D_uF^j)(u)\xi_j,\quad K(\xi)=K(0,\xi)$$
and
$$\mB(\xi)=\sum_{j=1}^d b^j\xi_j^2, ~~\xi \in \R^d.$$ 

We assume $F(0)=0$ w.l.o.g such that $(u,v)=(0,0)$ is a rest state of \eqref{xc1}.

Proposition \ref{prop:xc} is then a consequence of the following three results.

\begin{lemma}
	\label{lem:jx3}
	\eqref{xc1} satisfies (H) and (RH).
\end{lemma}

\begin{lemma}
	\label{prop:jx2}
	\begin{enumerate}
		\item[(i)]\eqref{xc1} satisfies (D1) if and only if 
		\begin{enumerate}
			\item[(D1)$_2$]
			all solutions $(\lambda,\xi) \in \C \times \R^d\setminus\{0\}$ of 
			\begin{equation}
				\label{det}
				\det(\lambda^2I_m+\mB(\xi)+iK(\xi)+\lambda I_m)=0.
			\end{equation}
			satisfy $\Rep \lambda <0$.
		\end{enumerate}
		\item[(ii)] \eqref{xc1} satisfies (D2) if and only if 
		\begin{enumerate}
			\item[(D2)$_2$]
			$(-\mathcal B(\omega)+K(\omega)^2)_{\omega \in S^{d-1}}$ is stably compatible with $(-iK(\omega))_{\omega \in S^{d-1}}$.
		\end{enumerate}
		\item[(iii)] \eqref{xc1} satisfies (D3)$'$ if and only if
		\begin{enumerate}
			\item[(D3)$_2$]
			the matrices $(-I_m \pm \mathcal B(\omega)^{-\frac12}K(\omega))_{\omega \in S^{d-1}}$  are stably compatible with $(i\mB(\omega))_{\omega \in S^{d-1}}$.
		\end{enumerate}
		
	\end{enumerate}
\end{lemma}

\begin{lemma}
	\label{lem:jx}
	If $b^j$ is a multiple of  the identity matrix. Then the dissipation condition
	$$\Rep(\delta_{jk}b^j-D_uF^jD_uF^k)_{1\le j,k \le d} >0$$
	implies (D1)$_2$, (D2)$_2$ and (D3)$_2$.
\end{lemma}

\begin{proof}[Proof of Lemma \ref{lem:jx3}]
	As $b^1,\ldots,b^d$ are positive definite $\mS=\operatorname{diag}(I_m,b^1,\ldots,b^d)$ is a symmetrizer for \eqref{xc}. (RH) is also trivially satisfied since $q_v=-I_{md}$
\end{proof}

\begin{proof}[Proof of Lemma \ref{prop:jx2}]
	(i) Let $(u,v) \in \R^m \times \R^{md}$ be an eigenvector to 
	$$\mM(\xi)=-i\tilde A(\xi)+D_uQ(0),$$ i.e. for $\xi=r\omega$, $r>0, \omega 
	\in S^{d-1}$
	\begin{align}
		\label{ew1}
		\lambda u+ir\Omega^tv&=0\\
		\label{ew2}
		(\lambda+I) v+(irB(\omega)-D_uF)u&=0.
	\end{align}
	One checks immediately that $\lambda=-1$ is an eigenvalue of multiplicity (at least) $m(d-1)$ with eigenvectors $(0,v)$, $v \in \ker \Omega^t$. On the other hand the eigenvalues with eigenvectors $(u,v)$, $u \neq 0$ are exactly the $2m$ solutions of \eqref{det} (counting multiplicities). Indeed if $u$ is in the kernel of $\lambda^2I_m+|\xi|^2\mB(\omega)+iK(\xi)+\lambda I_m$ then either $\lambda \neq -1$ and $(u,v)$ is an eigenvector with
	$v=(\lambda+1)^{-1}(-irA(\omega)+D_uF(0))u$
	or $\lambda=-1$ and $(u,v)$ is an eigenvector with $v$ in the kernel of
	$(irB(\omega)-D_uF(0))\Omega^t.$
	
	(ii) To treat (D2) we have to bring \eqref{xc} in normal form. To this end we apply the basis transformation 
	$$T=\begin{pmatrix}
		I_m & 0\\
		D_uF & I_{md}
	\end{pmatrix}$$
	and set
	\begin{equation} 
		\label{tildea}
		A^j=T^{-1}\tilde A^jT=\begin{pmatrix}
			(E_j)^tD_uF &(E^j)^t \\B^j-D_uF(E_j)^tD_uF & -D_uF(E_j)^t
		\end{pmatrix},~~\mL=T^{-1}D_UQ(0)T=\begin{pmatrix}
			0 & 0\\
			0 & -I_{md}
		\end{pmatrix}.
	\end{equation}
	In particular
	$$A_{11}(\omega)=\Omega^tD_uF=K(\omega)$$ and
	$$A_{12}(\omega)L ^{-1}A_{12}(\omega)=-\Omega^t(B(\omega)-D_uF\Omega^tD_uF)=-\mB(\omega)+K(\omega)^2.$$
	Thus the assertion follows by definition of (D2).
	
	(iii) With a slight abuse of notation we set
	$$B(\omega)^{\frac12}=\begin{pmatrix}
		(b^1)^{\frac12}\omega_1 & \cdots & (b^d)^{\frac12}\omega_j
	\end{pmatrix}^t.$$
	Note that $$(B(\omega)^{\frac12})^tB(\omega)^{\frac12}=\mB(\omega).$$
	
	First, perform a basis transformation with 
	$$\bar T=\operatorname{diag}(I_m,(b^1)^{\frac12},\ldots,(b^d)^{\frac12})$$
	and consider
	\begin{align*}\bar A(\omega)=\bar T^{-1}A\bar T&=\begin{pmatrix}
			0 & (B(\omega)^{\frac12})^t \\
			B(\omega)^{\frac12} & 0
		\end{pmatrix}\\
		\bar \mL=\bar TD_uQ(0)\bar T^{-1}&=\begin{pmatrix} 0 & 0\\ H^{-\frac12}D_uF & -I_m \end{pmatrix},\quad H=\operatorname{diag}(b^{1},\ldots,b^d).
	\end{align*}
	Next, fix $\bar \omega \in S^{d-1}$. For $\omega$ in neighbourhood of $\bar \omega$ let $\{W^1(\omega), \ldots, W^{m(d-1)}(\omega)\} \subset \R^{md}$ be a smooth family of orthonormal bases for $\ker \Omega^t$  and denote by $W(\omega)$ the matrix with columns $W^1(\omega),\ldots,W^{m(d-1)}(\omega)$. 
	Then for the orthonormal matrix
	$$\mT=\begin{pmatrix}
		0_{m \times m(d-1)} & \frac1{\sqrt2}I_m & \frac1{\sqrt2}I_m\\
		W(\omega) & B(\omega)^{\frac12} (2\mB(\omega))^{-\frac12} & B(\omega)^{\frac12} (2\mB(\omega))^{-\frac12}
	\end{pmatrix}$$
	we get
	$$\mT^{-1}(\omega)\bar A(\omega)\mT(\omega)=\begin{pmatrix} 0_{m(d-1)\times m(d-1)} & 0 & 0\\
		0 & \mathcal B(\omega)^{\frac12} & 0\\
		0 & 0 & -\mathcal B(\omega)^\frac12
	\end{pmatrix}$$
	as well as
	\begin{align*}\mT^{-1}\bar \mL\mT
		&=\begin{pmatrix}
			-I_{m(d-1)\times m(d-1)} & \ast & \ast \\
			\ast &\frac12(\mathcal B(\omega)^{-\frac12}K(\omega)-I_m) & \ast\\
			\ast & \ast &\frac12(-\mathcal B(\omega)^{-\frac12}K(\omega)-I_m)
		\end{pmatrix}.
	\end{align*}
	Hence $\mT^{-1}\bar \mL\mT$ is stably compatible with $\mT^{-1}A\mT$ if and only if (D3)$_2$ holds and (D3) is equivalent to (D3)$_2$ by basis invariance. 
\end{proof}

\begin{proof}[Proof of Lemma \ref{lem:jx}]
	If $b^j=\beta^jI_m$ for some $\beta^j \in \R$ we also have
	$$\mB(\omega)=\beta(\omega)I_m,\quad \beta(\omega)=\sum_{j=1}^d\beta^j\omega_j^2.$$
	Now set
	$$\mH:=(\delta_{jk}b^j-D_uF^jD_uF^k)_{1\le j,k \le d}.$$
	and let 
	\begin{equation} 
		\label{h}
		\Rep \mH>0.
	\end{equation} 
	Clearly, $W^*\mH W>0$ for all $W \in \C^{md}\setminus\{0\}$ implies
	$$0<\sum_{j,k=1}w^t(\delta_{jk}b^j-D_uF^jD_uF^k)\omega_j\omega_kw=w^*(\beta(\omega)I_m-K(\omega)^2)w$$
	for all $\omega \in S^{d-1}, w \in \R^m\setminus\{0\}$, i.e. 
	\begin{equation} 
		\label{posdef}
	\mB(\omega)-\Rep K(\omega)^2>0.
	\end{equation}
	This is clearly a special case of (D2). 

	Let $s(u)_{u \in \mU}$ be a family of Friedrichs symmetrizer for \eqref{cl}. \eqref{posdef} also implies that each eigenvalue $\mu^2$ of $K(\omega)^2$ satisfies $\Rep \mu^2 < \beta(\omega)$.  As $K(u,\omega)=\sum_{j=1}^dD_uF^j\omega_j$ is symmetrizable all eigenvalues of $K(\omega)$ are real and thus all eigenvalues $\mu$ of $K(\omega)$ satisfy $|\mu| <\beta^{\frac12}$. Since $$\tilde K(\omega)=s(0)^{\frac12}K(\omega)s(0)^{-\frac12}$$ is symmetric this yields
	\begin{equation} 
	\label{posdef2}
	\beta(\omega)^\frac12I_m\pm \tilde K(\omega)>0,
	\end{equation}
	which gives (D3).
	
	As (D2) and (D3) are already proven (D1) holds if  all
	solutions $(\lambda,\xi) \in \C \times \R^{d}\setminus\{0\}$ to \eqref{det} satisfy $\lambda \notin i\R$. Suppose there are $\mu \in \R, \xi \in \R^{d}\setminus\{0\}$ and $ V \in S^{m-1}$ with
	$$((i\mu)^2+\mB(\xi)+iK(\xi)+i\mu I_m)V=0.$$
	Applying the basis transformation $s(0)^\frac12$ and taking the scalar product with $\tilde V=s(0)^{\frac12}V$ gives
	$$(\mu^2-\beta(\xi))=0$$
	and 
	$$\tilde V^*(\tilde K(\xi)+\mu I_m)\tilde V=0.$$
	Thus (w.l.o.g. assume $\mu>0$)
	$$\tilde V^*(\tilde K(\xi))+\beta^{\frac12})\tilde V=0,$$
	which contradicts \eqref{posdef2}
\end{proof}

\begin{proof}[Proof of Proposition \ref{prop:nd}]
	 Suppose that \eqref{xc} is entropy dissipative and let $A^j,\mL$ be defined by  \eqref{tildea}. As seen in the proof of Proposition \ref{prop:ed} there exist symmetric positive definite matrices $S^1 \in \R^{m \times m}, S^2 \in \R^{md \times md}$ such that with $S=\operatorname{diag}(S^1, S^2)$ all matrices $SA^j$ are symmetric. For $j=1$ we find 
	$$S^2D_uFE^1=\begin{pmatrix} S^2D_uF & 0_{md \times m(d-1)}\end{pmatrix}$$
	As we assumed this matrix to be symmetric we get
	\begin{equation} 
		\label{nd1}
		\sum_{l=1}^dS^2_{kl}D_uF^l=0,~~k=2,\ldots, d.
	\end{equation}
	Also 
	$$S^1(E^1)^t=\begin{pmatrix} S_1^1 & 0_{md \times m(d-1)} \end{pmatrix},$$
	where $S^1_1$ is the first row of $S^1$.
	This yields due to $S$, $SA^1$ being symmetric
	$$0=\sum_{l=1}^dS^2_{k1}b^1-\sum_{l=1}^dS^2_{kl}(D_uF(E^1)^t D_uF)_l)=S^2_{k1}B^1-D_uF^1\sum_{l=1}^dS^2_{kl}D_uF^l$$
	for $k=2,\ldots, d$. Hence \eqref{nd1} gives $S^2_{k1}B^1=0$ for such $k$. As $b^1$ is invertible we find $S^2_{k1}=0$, $k\neq 1$ and $S^2$ has the form
	$$S_2= \begin{pmatrix}
		S^2_{11}& 0_{m \times m(d-1)}\\
		0_{m(d-1) \times m} &S^2_{22}
	\end{pmatrix}$$
	for some symmetric positive definite matrices $S^2_{11} \in \R^{m \times m}$, $S_{22}^2 \in \R^{m(d-1) \times m(d-1)}$. Plugging this back in \eqref{nd1} gives
	$$S^2_{22}\begin{pmatrix} Duf^2 \\ \vdots \\ Duf^d \end{pmatrix}=0,$$
	i.e. $D_uF^k=0$ for all $k \neq 1$. Repeating this argumentation with $\omega=e_2$ also shows $D_uF^1=0.$
	
\end{proof}

Using the results of Proposition \ref{prop:jx2} we can finally proof Proposition \ref{prop:none}.

\begin{proof}[Proof of Proposition \ref{prop:none}]
	For
	$$F(u)=(u_1-u_2,u_2-u_1)^t, \quad K=D_uF(0)=\begin{pmatrix}
		1 &-1\\1&-1
	\end{pmatrix}$$
	 and  the matrix $b=\operatorname{diag}(\kappa_1,\kappa_2)$ for $\kappa_1 \neq \kappa_2 >0$ consider 
	 $$\tilde A=\begin{pmatrix}
	 	0 & I_2\\
	 	b & 0
	 \end{pmatrix},\quad \tilde \mL=\begin{pmatrix} 0& 0\\ D_uF(0) & -I_2\end{pmatrix}$$
 These matrices correspond to the (linear) Jin-Xin system
 \begin{align*}
 	u_t+v_x&=0,\\
 	v_t+bv_x&=-(v-f(u)).
 \end{align*}
 The coefficient matrices in normal form are given by 
 $$A=\begin{pmatrix}
 	K & I_m\\
 	b-K^2 & -K
 \end{pmatrix},\quad \mL=\begin{pmatrix}
 0 & 0 \\ 0 & -I_2
\end{pmatrix}.$$
Thus by Proposition \ref{prop:jx2} $A$ and $\mL$ satisfy (D1), (D2), (D3) if and only if $K$ and $B(\omega)=b$ satisfy (D1)$_2$, (D2)$_2$, (D3)$_2$ respectively. 
	 
It was  shown in \cite{FS} that  condition (D2)$_2$ is equivalent to $\kappa_1+\kappa_2 >8$ and (D3)$_2$ to $\kappa_1,\kappa_2>1$. Additionally (D2)$_2$ holds if
$$(\kappa_1+\kappa_2 >8 \land \kappa_1,\kappa_2 \ge 1)~~\lor~~ (\kappa_1+\kappa_2 \ge 8 \land \kappa_1,\kappa_2 > 1).$$  Thus for $\kappa_1=1$, $\kappa_2>7$  (D1) and (D2) are satisfied but (D3) is not. On the other hand, for  $1<\kappa_1<7$, $\kappa_2=8-\kappa_1$ (D1) and (D3) hold but (D2) does not. As $A_{11}$ and $A$ are constantly hyperbolic this shows Proposition \ref{prop:none} due to \ref{prop:hc2}. 
\end{proof}

\textbf{Connections to second-order-hyperbolic systems.} In the proof of Proposition \ref{prop:none} we made use of an example of \cite{FS}, which was constructed for second-order hyperbolic systems. This works due to the following fact: By differentiating the first equation of \eqref{xc} w.r.t. $t$, the equations for $v^j_t$ w.r.t. $x_j$, adding the latter and substracting the resulting equations we see that for smooth solutions $(u,v)$ of \eqref{xc} $u$ satisfies the second-order hyperbolic system
\begin{equation}
	\label{so}
	u_{tt}-\sum_{j=1}^d b^{j}u_{x_jx_j}+u_t+\sum_{j=1}^d D_uF^j(u)u_{x_j}=0.
\end{equation}
Now \cite{sro23} tells us that for $d\ge 3$ solutions $U=(u,u_t)$ to the corresponding Cauchy-problem are $(H^{s+1} \times H^s)\cap L^1 \to H^s \times H^{s-1}$, $s>d/2+1$ asymptotically stable with rate $(1+t)^{-\frac{d}4}$ if (D1)$_2$, (D2)$_2$, (D3)$_2$ hold. However, we could not have proven that \eqref{xc} satisfies Theorem \ref{main} by only refering to the results in \cite{sro23} since the smallness of the intial data $(u_0,v_0)$ in the $L^1$-norm does not imply the smallness of $u_t(0)=-\sum_{j=1}^d v^j_{x_j}(0)$ in the $L^1$-norm.

\appendix

\section{Stable Compatibility}

\begin{proof}[Proof of Lemma \ref{lem:ch}]
	First note that (SC)$'$ is invariant under smooth basis transformations. Hence it holds for all projectors $\tilde J_{Rl}$ if it holds for one projector $J_{Rl}$.
	
	The necessity of (SC)$'$ is obvious even in the case of non-constant multiplicities. 
	
	Regarding its sufficiency fix $\bar \omega \in S^{d-1}$, $l \in \{1,\ldots,\bar l\}$ and let $\tilde J_{Rl}: \mathcal V_{\bar \omega} \to \C^{n \times \alpha_{l}}$ be a right projector on $R(P_l(\omega))$. Due to constant multiplicity
	$$\tilde J_{Ll}(\omega)H(\omega)\tilde J_{Rl}(\omega)=\mu_l(\omega)I_{\alpha_l}$$
	As all eigenvalues of 
	$$h_l(\omega)= \tilde J_{Ll}( \omega) h(\omega) \tilde J_{Rl}(\omega)$$ have strictly negative real part there exists an invertible matrix $T=T(\bar \omega) \in \C^{\alpha_l \times \alpha_l}$ such that 
	$$\Rep (T^{-1} h_l(\bar \omega) T) \le -cI_{\alpha_l}.$$
	Since 
	$$T^{-1}\tilde J_{Ll}(\omega)H(\omega)\tilde J_{Rl}(\omega)T=\mu_l(\omega)I_{\alpha_l}$$
	with $\mu \in i\R$, \eqref{proj} follows with $J_{Rl}=\tilde J_{Rl}T$. 
\end{proof}

\begin{proof}[Proof of Lemma \ref{lem:sym}]
	Since $H(\omega)$ is anti-Hermitian, clearly also $J_{Ll}(\omega)H(\omega)J_{Rl}(\omega)$ is anti-Hermitian if $J_{Ll}(\omega)=J_{Rl}(\omega)^*$. The result then follows directly from the fact that a Hermitian matrix $K$ is negative definite if and only if all of its 
	eigenvalues have strictly negative real part.
\end{proof}

\begin{proof}[Proof of Lemma \ref{lem:sc}]
	The lemma is a generalization of Lemma 5 in \cite{FS} and we follow the argumentation therein. 
	
	First, fix $\bar \omega \in \Omega$. Below we  $(\kappa,\omega) \in P_{\bar \omega}$ for a neighbourhood $P_{\bar \omega} \subset [0,\infty) \times \Omega$ of $(0,\bar \omega)$, which may be chosen sufficiently small (in the sense of inclusion of sets) in each step. Set
	$$T_0:=\begin{pmatrix}
		J_{R1} & \ldots  J_{R\bar l}
	\end{pmatrix},~~ H_l:=J_{Ll}H(\omega)J_{Rl},~l=1,\ldots,\bar l.$$
	Then by definition 
	$$T_0^{-1}HT_0=\operatorname{diag}(H_1,\ldots,H_{\bar l}).$$
	 By spectral separation, there exists a smooth family  of basis transformations $T:P_{\bar \omega} \to \C^{n \times n}$ with $T(0,\omega)=T_0(\omega)$ such that
	$$T(\kappa,\omega)^{-1}\mathcal H(\kappa,\omega)T(\kappa,\omega)=\operatorname{diag}(\mathcal H_1(\kappa,\omega),\ldots,\mathcal H_{\bar k}(\kappa,\omega))$$
	with blocks $\mathcal H_l(\kappa,\omega) \in \C^{\alpha_l \times \alpha_l}$, $l=1,\ldots,\bar l$ satisfying 
	$$\mathcal H_l(0,\omega)=J_{Ll}(\omega)\mathcal H(\omega)J_{Rl}(\omega).$$ In particular $\mathcal H_l(0,\bar \omega)=\mu_l(\bar \omega) I_{\alpha_l}$, which gives
	$$(\partial_{\kappa}\mathcal H_l)(0,\bar \omega)=(\partial_{\kappa}(\mathcal H_l-\mu_lI_{\alpha_l}))(0,\bar \omega)=J_{Ll}(\bar \omega)h(\bar \omega)J_{Rl}(\bar \omega).$$
	As $h$ is stably compatible with $H$ this implies
	$$\Rep\mathcal H_l(0,\omega)= 0,\quad \Rep \partial_\kappa \mathcal H_l(0,\omega) <0.$$ 
	In conclusion
	$$\Rep T(\kappa,\omega)^{-1}\mathcal H(\kappa,\omega)T(\kappa,\omega)\le - c\kappa I_n$$
	for all $(\kappa,\omega) \in P_{\bar \omega}$,  and it is straightforward to see that
	$$\mathcal D=(T^{-1})^*T^{-1}$$
	satisfies \eqref{S1}, \eqref{S2} restricted to $P_{\bar \omega}$.
	
	Now, let $K$ be a compact subset of $\Omega$. Then we find $\omega_1,\ldots,\omega_{\bar k}$ and, for each $k=1,\ldots,\bar k$ a neighbourhood $P_{\bar \omega}$ and a positive constants $C_k$ $c_k$,   as well as smooth matrix families $\mathcal D_k:\overline{P_{\omega_k}} \to \C^{N \times N}$ such that $\mathcal D_k$ satisfies \eqref{S1}, \eqref{S2} (with constants $C=C_k,c=c_k$) and for some $\kappa_0>0$ the family $(P_{\bar \omega_k})_{1 \le k \le \bar k}$
	is an open cover of $[0,\kappa_0] \times K$. Now choose a partition of unity $(\zeta_k)_{1 \le k \le \bar k}$ subordinate to this cover. Extend each $\mathcal D_k$ trivially by $0$ outside $P_{\bar \omega_k}$. Then it is straightforward to show that \eqref{S1}, \eqref{S2} hold for
	$$\mathcal D=\sum_{k=1}^m \mathcal D_k\zeta_k,~ c=\min_{k=1,\dots,m}c_k,~C=\sum_{k=1}^m C_k.$$
\end{proof}

\section{Key ideas regarding non-linear stability in hyperbolic balance laws}

To illustrate the main ideas of the present work, in particular in comparison to \cite{kawa83,kawa85,yon01,kawa04,yon04,bia07,kawa09} consider the one dimensional case $d=1$, the linearization of \eqref{eq:haupt} about an equilibrium state $\bar U \in \mathcal U$: 
\begin{equation} 
	\label{lin0}
	U_t+AU_{x}=\mL U,
\end{equation} 
$A=Df^1(\bar U), \mL=D_UQ(\bar U) \in \R^{n \times n}$, and its Fourier-transformed version
\begin{equation} 
	\label{linf0}
	\hat U_t+i\xi A\hat U=\mL\hat U,~~\xi \in \R.
\end{equation} 
In any of the works above it is crucial to establish a decay estimate of the form
\begin{equation} 
	\label{dec0}
	|\hat U(t,\xi)|^2 \le C\exp(-c\rho(\xi)t)|U(0,\xi)|^2,~~(t,\xi) \in [0,\infty) \times \R,
\end{equation}
or equivalently, for $\mM(\xi)=-i\xi A+\mL$.
\begin{equation}
	\label{dec1}
	|\exp(\mM(\xi)t)| \le C\exp(-c\rho(\xi)t)
\end{equation}
where the rate $\rho(\xi)$ is given as
\begin{equation}
	\label{rho}
	\rho(\xi)=\frac{\xi^2}{1+\xi^2}.
\end{equation}
Since in leading order $\mM(\xi)$
behaves as $\mL$ and $iA$ for $|\xi| \to 0$ and $|\xi| \to \infty$ respectively, for such an estimate to hold it is necessary that the eigenvalues  of $A$ are real and semi-simple and those of $\mL$ have real-part less than or equal to $0$. In addition, if $\mL$ has an eigenvalue with vanishing real part the decay \eqref{dec0} is optimal. Thus we say \eqref{lin0} has optimal decay if an estimate of the form \eqref{dec0} holds. Following \cite{FS, sro23} we also call such systems uniformly  dissipative.

For simplicity (and as this is always the case in applications) assume that $\mL$ possesses no purely imaginary eigenvalue.  Then the necessary conditions above imply the existence of basis transformations $S_A$ and $S_L$ such that 
\begin{equation} 
	\label{a00}
	S_AAS_A^{-1}~~ \text{is symmetric}
\end{equation}
and
\begin{equation} 
	\label{l0}
	S_L\mL S_L^{-1}=\begin{pmatrix}
		0 & 0\\
		0 & \tilde L
	\end{pmatrix}~~\text{with}~~\Rep \tilde L=\frac12(\tilde L+\tilde L)^*>0.
\end{equation}
(Here `$\cdot>0$' means positive definiteness of the matrix.)  Now, the existence of an entropy and the entropy dissipation conditions assumed in \cite{yon04} (Section 1, conditions (i)-(iii)) crucially imply that one can choose $S_A=S_L$ in \eqref{a00}, \eqref{l0} (cf. the proof of Proposition \ref{prop:ed} or the argumentation in \cite{bia07}). Thus in variables $V=S^{-1}U$ system \eqref{lin0} can be written in the form
\begin{equation}
	\label{cd}
	V_t+\tilde A V_x=\tilde{\mL}V,~\tilde A~~\text{symmetric},~\tilde{\mL}=\begin{pmatrix}
		0 & 0\\
		0 & \tilde L
	\end{pmatrix},~~\Rep \tilde L>0. 
\end{equation}
Following the terminology of \cite{bia07} we say that such a system has conservative-dissipative form (C-D form). In this case it can be shown (\cite{kawa85,bia07}) that optimal decay is equivalent to the Kawashima-Shizuta condition
\begin{enumerate}
	\item[(K)] $\forall \lambda \in \C:\ker(\lambda+A) \cap \ker \mL=\{0\},$
\end{enumerate}
which in turn is equivalent to the decay of all Fourier-Laplace-modes in the sense that all eigenvalues of $-i\xi A+\mL$ have strictly negative real-part for $\xi \in \R^{d} \setminus\{0\}.$ In addition, (K) implies  asymptotic $H^s\cap L^1 \to H^{s-1}$ stability with rate $(1+t)^{-\frac{d}4}$ for $s\ge [d/2]+2$. This result  generalizes in a straightforward manner for the multi-D case. 

In the present work we neither assume simultaneous symmetrizability of $A$ and $\mL$ nor  the existence of an entropy as in \cite{yon04}, which would allow for writing the linearization in C-D form.  In this general case (K) is not a suitable condition to determine whether \eqref{lin0} admits optimal decay (Proposition \ref{prop:none}). However, we  give sufficient and (under the assumption of constant multiplicity) necessary algebraic (dissipativity) conditions for optimal decay and show that these imply (non-linear) $H^s \cap L^1 \to H^{s-1}$  stability, $s\ge d/2+2$, with rate $(1+t)^{-d/4}$ for $d\ge 3$. In order to motivate these conditions, note that optimal decay is equivalent to the existence of a family of basis-transformations $(S(\xi))_{\xi \in \R}$ such that $|S(\xi)|, |S(\xi)^{-1}|$ are uniformly bounded and
\begin{equation} 
	\label{dec2}
	\Rep S(\xi)\mM(\xi)S(\xi)^{-1} \le -c\rho(\xi)I_N ,~~\xi \in \R^d.
\end{equation}
By a matrix perturbation argument with respect to $1/|\xi|$ it is quite straightforward to show (Lemma \ref{lem:lw}) that \eqref{dec2} holds for $|\xi|$ sufficiently large if and only if
\begin{enumerate}
	\item[(K')] For all eigenvalues $\lambda$ of $A$ there exists a basis $(r_1,\ldots,r_m)$  of the eigenspace corresponding to $\lambda$ with dual basis $(p_1,\ldots,p_d)$ such that with
	$$R=(r_1,\ldots,r_m),\quad P=\begin{pmatrix} 
		p_1\\ \vdots \\p_m\end{pmatrix}$$
	there holds
	$$\Rep P\mL R <0.$$
\end{enumerate}
For systems in C-D form  $\Rep \tilde{\mL}$ is positive semi-definite and $\ker \tilde{\mL}=\ker \Rep \tilde{\mL}$. Therefore (K') is already implied by (K). However,  it turns out (Proposition \ref{prop:none}) that in general (K) or even (K') is not sufficient for optimal decay and the additional conditions (D1), (D2) are necessary in order to obtain \eqref{dec0}.

Lastly, we would like to point our that the difficulties in treating more general systems are not restricted to the linear level. Indeed, for linear systems which can be written in C-D form (K) is equivalent to the existence of a so-called compensating matrix (cf. \cite{kawa83}, \cite{bia07}). This matrix, which only depends on the rest state, can then be used to obtain the central a-priori decay estimates also for the non-linear system. In our general framework, such a matrix is not accessible. Instead to treat non-linear perturbation we construct a certain symmetrizer depending on the perturbation itself, which necessitates the use of para-differential calculus. The technical details of such a construction have been laid out in \cite{sro23}.

\section*{Statements and declarations} 

\textbf{Declaration of interests.} None.

\textbf{Acknowledgement.} The author would like to sincerely thank Heinrich Freistühler for his highly helpful suggestions and comments as well as many fruitful discussions.

\bibliographystyle{abbrvnat}

\end{document}